\documentclass[11pt]{amsart}
\usepackage{amssymb,amsfonts,amsthm,amsmath,enumerate,amscd}
\usepackage[english]{babel}
\usepackage{colordvi,color}

\newtheorem{theorem}{Theorem}

\newtheorem{lemma}[theorem]{Lemma}
\newtheorem{corollary}[theorem]{Corollary}
\newtheorem{claim}[theorem]{Claim}
\newtheorem{proposition}[theorem]{Proposition}
\newtheorem{definition}[theorem]{Definition}

\newtheorem{remark}[theorem]{Remark}

\newcommand{\field}[1]{\mathbb{#1}}

\newcommand{\R}{\field{R}}

\newcommand{\wt}[1]{\widehat{#1}}
\newcommand{\ba}{{\bf a}}

\newcommand{\bb}{{\bf b}}
\newcommand{\bbo}{{\bf 0}}
\newcommand{\bB}{{\bf B}}

\newcommand{\bc}{{\bf c}}

\newcommand{\bS}{{\bf S}}

\newcommand{\bu}{{\bf u}}
\newcommand{\bU}{{\bf U}}

\newcommand{\bw}{{\bf w}}

\newcommand{\bx}{{\bf x}}

\newcommand{\by}{{\bf y}}
\newcommand{\bz}{{\bf z}}

\newcommand{\cD}{\mathcal{D}}

\newcommand{\cL}{\mathcal{L}}
\newcommand{\cM}{\mathcal{M}}

\newcommand{\clos}{{\rm clos}}

\newcommand{\dlt}{{\delta}}
\newcommand{\Dlt}{{\Delta}}

\newcommand{\gm}{{\gamma}}

\newcommand{\iphi}{{\iota(\phi)}}
\newcommand{\iPhi}{{\iota(\Phi)}}
\newcommand{\ipsi}{{\iota(\psi)}}

\newcommand{\lsm}{{\lesssim}}
\newcommand{\omg}{\omega}

\newcommand{\Rp}{{\R^p}}
\newcommand{\Rpbar}{\overline{\R^p}}
\newcommand{\Rqbar}{\overline{\R^q}}
\newcommand{\Rq}{{\R^q}}

\newcommand{\sgm}{\sigma}
\newcommand{\tht}{{\theta}}
\newcommand{\ups}{{\upsilon}}

\newcommand{\vp}{\varphi}
\newcommand{\wtK}{\widetilde{K}}
\newcommand{\wtPhi}{\widetilde{\Phi}}
\newcommand{\wtS}{\widetilde{S}}
\newcommand{\wtW}{\widetilde{W}}

\numberwithin{equation}{section}

\numberwithin{equation}{section}

\begin{document}
\title[]{Stereographic compactification and affine bi-Lipschitz homeomorphisms.}

\author[V. Grandjean]{Vincent Grandjean}
\thanks{We are very grateful to Andr\'e Costa and Maria Michalska for 
conversations, comments and insight.}
\address{V. Grandjean, Departamento de Matem\'atica, 
Universidade Federal de Santa Catarina, 
88.040-900 Florianópolis - SC, Brasil}
\email{vincent.grandjean@ufsc.br}
%
\author[R. Oliveira]{Roger Oliveira}
\address{R. Oliveira, Faculdade de Educa\c{c}\~ao, Ci\^encias e Letras do Sert\~ao Central,
Planalto Universit\'ario, Quixad\'a - CE, 63900-000}
\email{roger.oliveira@uece.br}



\begin{abstract}
Let $\sgm_q : \Rq \to \bS^q\setminus N_q$ be the inverse of the stereographic 
projection with centre the north pole $N_q$. Let $W_i$ be a closed subset of 
$\R^{q_i}$,
for $i=1,2$. Let $\Phi:W_1 \to W_2$ be a bi-Lipschitz homeomorphism.
The main result 
states that the homeomorphism $\sgm_{q_2}\circ \Phi \circ \sgm_{q_1}^{-1}$ 
is a bi-Lipschitz homeomorphism, extending bi-Lipschitz-ly at $N_{q_1}$ with
value $N_{q_2}$ whenever $W_1$ is unbounded. 

As two straightforward 
applications in the 
polynomially bounded o-minimal context over the real numbers, we obtain for
free a version at infinity of: 1) Sampaio's tangent cone result; 2) Links 
preserving re-parametrization of definable bi-Lipschitz homeomorphisms of 
Valette. 
\end{abstract}

\maketitle



\section{Introduction}

\medskip
Any subset $S$ of $\Rq$ is equipped with the \em outer metric structure, \em 
where the distance between points of $S$ is their (Euclidean) distance in $\Rq$. 
Thus  (outer) Lipschitz mappings $S_1 \to S_2$, for 
$S_i$ a subset of $\R^{q_i}$, are well defined.

\medskip
Given a positive integer $q$, let $N_q = (0,\ldots,0,1) \in\R^q\times\R$, 
the north pole of the unit sphere $\bS^q$ of $\R^{q+1}$. The inverse of the 
\em stereographic projection of $\bS^q$ onto $\Rq$ with centre $N_q$ \em
is denoted $\sgm_q : \R^q \to \bS^q\setminus N_q$.
Let $\wtS$ be the closure in $\bS^q$ of $\sgm_q(S)$ where $S$ is a closed 
subset of $\Rq$. 

\medskip
For $i=1,2$, let $W_i$ be a closed subset of $\R^{q_i}$.
 If $\Phi :W_1 \to W_2$ is a homeomorphism, the \em 
stereographic pre-compactification of $\Phi$ \em is the following homeomorphism 
$$
\sgm_{q_2}\circ\Phi\circ \sgm_{q_1}^{-1} : \sgm_{q_1}(W_1) \to 
\sgm_{q_2}(W_2). 
$$
Since $W_1,W_2$ are closed, we observe that the stereographic pre-compactification 
of $\Phi$ extends as a homeomorphism $\wtPhi:\wtW_1 \to \wtW_2$ mapping 
$N_{q_1}$ onto $N_{q_2}$ whenever $W_1$ is unbounded. 
We call this extension the \em stereographic compactification of $\Phi$. \em
If $W_1$ is compact, so is $W_2$ and the stereographic pre-compactification of $\Phi$ is 
its stereographic compactification. 

\medskip
The main 
result of this note is the
following

\medskip\noindent
{\bf Theorem \ref{thm:main}.} \em The mapping $\Phi$ is bi-Lipschitz if
and only if its stereographic compactification $\wtPhi$ is bi-Lipschitz. \em

\medskip
The main result is a consequence of Lemma \ref{lem:main} presented below. 
We recall that the \em Euclidean inversion of $\R^q$ \em is the
following mapping 
$$
\iota_q : \R^q\setminus \bbo \to \R^q \setminus \bbo, \;\;
\bx \; \mapsto \; \frac{\bx}{|\bx|^2}.
$$
Let $\Phi:W_1 \to W_2$ be a homeomorphism between the closed subsets $W_i$
of $\R^{q_i}$, $i=1,2$. 
The \em inversion of the mapping  $\Phi:W_1 \to W_2$ \em is defined as follows
$$
\iPhi : \iota_{q_1}(W_1\setminus \bbo) \to  \iota_{q_2}(W_2\setminus \bbo), 
\;\; \bx \mapsto \iPhi(\bx) :=  \iota_{q_2} \circ \Phi \circ \iota_{q_1}^{-1} (\bx).
$$

\smallskip
\noindent
The next result (so called \em the inversion lemma\em) is the main tool we use to obtain the main result. It is
most certainly of interest on its own, and can be applied in many different 
contexts. 

\medskip\noindent
{\bf Lemma \ref{lem:main}.} \em Assume furthermore that,
either $W_i$ contains the origin $\bbo\in \R^{q_i}$ for $i=1,2$ and 
$\Phi(\bbo) = \bbo$, or $\bbo\notin 
W_i$ for $i=1,2$.

The homeomorphism $\iPhi$ is bi-Lipschitz
if and only $\Phi$ is. 
Moreover, if $W_1$ is unbounded, then $\iPhi$ extends bi-Lipschitz-ly at 
$\bbo$ taking the value $\bbo$. \em

\bigskip
Our interest in this problem arose from results of the recent PhD 
Thesis of the second named author \cite{Oli}. Indeed, the bi-Lipschitz 
classification of local plane objects germs, at the origin, respectively at 
infinity and in correspondence by the Euclidean inversion, presented 
strikingly similar properties, now explained by Lemma \ref{lem:main}.
\\
The main result is a convenient reformulation
of the inversion lemma, and is also very much in tune with the joint works 
of the first named author \cite{CoGrMi1,CoGrMi2} which expand 
the results of the recent PhD Thesis~\cite{Cos}.

\bigskip
The paper is organised as follows: Section 2 introduces various 
preliminary materials and notations. Section \ref{section:IMaGbLH}
presents the special case of a global bi-Lipschitz homeomorphism of $\Rq$. 
Section 
\ref{section:IaGobLHaI} and Section \ref{section:IaGobLHa0}, respectively, 
show germ-ified versions of the inversion lemma, namely Lemma 
\ref{lem:magic-infinity} at $\infty$ and, respectively  Lemma 
\ref{lem:magic-origin} at $\bbo$. Section 
\ref{section:IabLH} is the short proof of our main tool, the inversion lemma 
\ref{lem:main}. The main result is dealt with in Section \ref{section:MR}.
The last section presents two immediate applications, 
versions at infinity of two results about germs of definable subsets at 
the origin: Proposition \ref{prop:tangent-cones-infty} (Sampaio's tangent 
cone result \cite{Sam}) and Proposition \ref{prop:bilip-dist-pres-infty} 
(Valette's Links preserving re parametrization of definable bi-Lipschitz 
homeomorphism \cite{Val2}).
%
%
%
%
%
%
%
%
%
%
%
%
%
%
%
%
%
%
%
%
\section{Preliminaries}\label{section:P}
\subsection{Notations}\label{subsec:notations}
$ $

The Euclidean space $\Rq$ is equipped with the Euclidean distance, denoted
$|-|$. We denote by $B_r^q$ the open ball of $\Rq$ of radius $r$ and centred 
at the origin $\bbo$, by $\bB_r^q$ its closure and by $\bS_r^{q-1}$ its 
boundary. The open ball of radius $r$ and with centre $\bx_0\in\Rq$ is 
$B^q(\bx_0,r)$, its closure is $\bB^q(\bx_0,r)$ and $\bS^{q-1}(\bx_0,r)$ is 
its boundary. The unit sphere $\bS_1^{q-1}$ is simply denoted by $\bS^{q-1}$.

If $S$ is any subset of $\Rq$, its closure in $\Rq$ is $\clos(S)$, and $S^*$ 
is $S\setminus \bbo$.

\bigskip
Let $\bU_q$ be the punctured affine space $\Rq^*$.

Compactifying the space $\Rq$ with the point $\infty$ at infinity as
$$
\Rqbar := \Rq\sqcup \infty = \bbo \sqcup \bU_q \sqcup \infty
$$
yields a space that is smoothly diffeomorphic to the unit 
sphere $\bS^q$ of $\R^{q+1}$, using the stereographic projections centred at
the "north" and "south" poles of $\bS^q$. Under this correspondence, the 
points $\bbo$ and $\infty$ are antipodal.

If $S$ is any subset of $\Rq$ its closure in $\Rqbar$ is $\overline{S}$.
Thus $S$ is unbounded if and only if $\overline{S} = \clos(S) \cup \infty$.

\bigskip
The germ $(\Rq,\infty)$ of $\Rq$ at infinity  is well defined, and can be 
considered as a germ in $\Rq$ as well as in $\Rqbar$. 

Let $\gm$ be a point of $\Rqbar$. Let $(\bx_n)_n, (\by_n)_n$ be two sequences 
of $\Rq$ converging to $\gm$ in $\Rqbar$. Let $\bz_n$ be $\bx_n$ or $\by_n$ 
and let 
$$
z_n := 
\left\{
\begin{array}{rcl}
|\bz_n - \gm| & {\rm if} & \gm \in \Rq \\
|\bz_n| & {\rm if} & \gm = \infty
\end{array}
\right.
$$
We will use the following notation
$$
\bx_n \sim \by_n  \; \Longleftrightarrow \; \frac{x_n}{y_n} 
\in [a,b] \;\; {\rm for} \;\; a,b>0 \;\; {\rm whenever} \;\; n\gg 1,
$$
as well as the next one 
$$
\bx_n \; \lsm \; \by_n  \; \Longleftrightarrow \; \frac{x_n}{y_n} 
\in [0,a] \;\; {\rm for} \;\; a>0 \;\; {\rm whenever} \;\; n\gg 1,
$$
and the last one 
$$
\bx_n = o(\by_n)  \; \Longleftrightarrow \; \lim_n \frac{x_n}{y_n} = 0.  
$$

\smallskip
\subsection{On affine subsets}
$ $ 

Any non-empty subset $S$ of $\Rq$ inherits from the ambient Euclidean 
structure of $\Rq$ the \em outer metric space structure \em $(S,d_S)$, where 
$$
d_S(\bx,\bx') : = |\bx - \bx'|
$$ 
for any pair of points $\bx,\bx'$ of $S$. 
We recall that if a mapping 
$\vp: (S,d_S) \to \Rp$ is Lipschitz with Lipschitz constant $C$ it extends 
as a Lipschitz mapping $(\clos(S),d_{\clos(S)}) \to \Rp$ with the same
Lipschitz constant $C$. In practise we can assume that $S$ is closed in $\Rq$.

In order to ease accumulation of hypotheses and notations, we introduce the 
following
\begin{definition}\label{def:q-affine}
A \em $q$-affine subset \em is a non-empty closed subset of $\Rq$ with $q\geq 1$. 

An \em affine subset \em is a $q$-affine subset for some positive integer $q$.

\end{definition}
Since any affine subset $S$ is equipped with the outer metric space 
structure $(S,d_S)$ described above, we introduce the following
\begin{definition}\label{def:affine-lip}
A Lipschitz mapping $S \to T$ between the affine subsets $S,T$ is a Lipschitz 
mapping $(S,d_S) \to (T,d_T)$. 
\end{definition}

\medskip
\subsection{On the inversion}
$ $

The inversion of the (punctured) affine space $\Rq$, defined as
$$
\iota_q : \bU_q \to \bU_q, \;\; \bx \mapsto \frac{\bx}{|\bx|^2}
$$
is a $C^\infty$ (semi-algebraic) diffeomorphism, and extends as a (semi-algebraic) homeomorphism ($C^\infty$ actually) over $\Rqbar$ exchanging the origin 
$\bbo$ and the point at infinity $\infty$. 

\medskip
Let $\bx$ be any point of $\bU_q$. Let 
$$
R(\bx) := \R\bx 
$$ be the real vector line through 
$\bx$.  
The tangent space of $\bU_q$ at $\bx$ decomposes as the Euclidean orthogonal sum
$$
T_\bx \bU_q = R(\bx) \oplus \bS(\bx) \;\; {\rm where} \;\; 
\bS(\bx) := T_\bx \bS_R^{q-1}.
$$
Observe that $\bS(s \bx) = \bS(\bx)$ and $R(s\bx) = R(\bx)$, 
as vector subspaces of $\Rq$, whenever $s\neq 0$.
An elementary computation shows that in the previous orthogonal basis of 
$T_\bx \bU_q$ we obtain  
$$
D_\bx \iota_q := -\frac{1}{|\bx|^2}Id_{R(\bx)} \oplus 
\frac{1}{|\bx|^2}Id_{\bS(\bx)} = \frac{1}{|\bx|^2}\left[ - Id_{R(\bx)} \oplus 
Id_{\bS(\bx)}\right]
$$
In particular $D_\bx\iota_q$ is an orthogonal mapping, since $|\bx|^2 
D_\bx\iota_q$ is simply the orthogonal symmetry w.r.t. the hyperplane
$\bS(\bx)$. We thus deduce the (Euclidean) norm of $D_\bx\iota_q$
\begin{equation}\label{eq:norm-inversion}
\|D_\bx \iota_q \| = \frac{1}{|\bx|^2}.
\end{equation}

\medskip
\subsection{Elementary, yet very useful, identities}
$ $

\smallskip
We recall the following known estimates
\begin{claim}\label{claim:lip-1}
Let $\bx,\bx' \in \Rq$ and $C>0$ such that $|\bx'| \geq (1+C) |\bx|$. Then
$$
\frac{C}{1+C}|\bx'| \, \leq \, |\bx' - \bx| \, \leq \, \frac{2+C}{1+C}|\bx'|.
$$
\end{claim}

\bigskip
\noindent
Given $\bx_1,\bx_2 \in \bU_q$, we define 
$$
e:= |\bx_1 - \bx_2|, \;\; r_i:= |\bx_i|, \;\; \by_i := \iota_q(\bx_i), \;\;
E:= |\by_1 - \by_2| \;\; {\rm and} \;\; R_i := |\by_i|,
\;\;  i=1,2.
$$ 
We assume that $r_1 = (1+C)r_2$ for some $C\geq 0$. Let $2\tht\in[0,\pi]$ be 
the angle between $\bx_1,\bx_2$ (thus between $\by_1,\by_2$ as well).
Then $r_1 -r_2 = Cr_2$ and $R_2 - R_1 = 
CR_1$. We recall that \em the law of cosines \em is the following identity 
\begin{equation}\label{eq:law-cosine}
e^2  = (r_1 - r_2)^2 \cos^2\tht + (r_1+r_2)^2 \sin^2\tht = 
r_2^2[C^2 +  4 (1+C)\sin^2\tht] .
\end{equation}
%
%
The inversion and the law of cosines give the following identity 
\begin{equation}\label{eq:e-E}
E = R_1R_2 \cdot e \;\; {\rm and} \;\; e = r_1 r_2 \cdot E.
\end{equation}
\section{Inversion mapping and global bi-Lipschitz homeomorphisms}\label{section:IMaGbLH}
\medskip
We present a special occurrence of the inversion lemma. 
Although it is likely that it has already been written in a few books, 
we give a proof, following from elementary Lipschitz analysis.

\medskip
Let $\cL(a,b)$ be the space of $\R$-linear mappings $\R^a\to\R^b$.

\medskip
Let $\vp:\Rp \to \Rq$ be a Lipschitz mapping with Lipschitz constant
$A_\vp$: 
$$
\bx,\bx'\in \Rp \; \Longrightarrow \; 
|\vp (\bx) - \vp(\bx')| \; \leq \; 
A_\vp\cdot |\bx - \bx'|.
$$
Let $\cD(\vp)$ be the set of points where $\vp$ is differentiable. Rademacher 
Theorem states that the complement $\Rp\setminus \cD(\vp)$ is of null 
measure \cite{Rad,Hei}. We consider the following closed subset
$$
\Dlt(\vp) := \clos\left(\left\{(\bx,D_\bx \vp) \in \cD(\vp) \times 
\cL(q,p)\right\}\right) \subset \Rp\times \cL(p,q).
$$
Let $\pi_\cL : \Rp\times \cL(p,q)\to\cL(p,q)$ be the projection onto the second
factor. Let 
$$
\cL(\vp) := \pi_\cL(\Dlt).
$$
For any $\bx\in \cD (\vp)$ the Lipschitz
condition on $\vp$ yields the following estimate about the norm
of $D_\bx \vp$
$$
\|D_\bx \vp\| \,  \leq  \, A_\vp.
$$
Since the norm is continuous over $\cL(q,p)$, 
we deduce that that
\begin{equation}\label{eq:norm-lipschitz}
L\in \cL(\vp) \; \Longrightarrow \; \|L \| \leq A_\vp.
\end{equation}

\bigskip
Let $H:\Rq \to \Rq$ be a bi-Lipschitz homeomorphism mapping
the origin onto itself, with Lipschitz constant $A_H>0$:
$$
\bx,\bx'\in \Rq \; \Longrightarrow \; 
\frac{1}{A_H}\cdot |\bx - \bx'| \; \leq \; |H(\bx) - H(\bx')| \; \leq \; 
A_H\cdot |\bx - \bx'|.
$$
Therefore from Estimate \eqref{eq:norm-lipschitz} we get
$$
L\in \cL(H) \; \Longrightarrow \; \frac{1}{A_H} \; \leq \; \| L \| \; \leq 
\; A_H.
$$
The mapping $H$ extends as a homeomorphism of $\Rqbar$ mapping $\infty$ onto 
$\infty$. 

\bigskip
The \em inversion of $H$ \em is the mapping $\iota_q \circ H \circ 
\iota_q^{-1}$. It is a homeomorphism of $\bU_q$ which extends continuously
to $\bbo$, taking the value $\bbo$, as the homeomorphism $\iota(H) :\Rq \to 
\Rq$. More precisely the following holds true
\begin{proposition}\label{prop:inversion-global}
The inversion $\iota(H)$ of $H$ is bi-Lipschitz.
\end{proposition}
\begin{proof}
It is enough to show that $\iota(H)$ is Lipschitz, since $\iota(H^{-1}) = 
\iota(H)^{-1}$. Let $\by\in \bU_q$ and 
let $\bx := \iota_q^{-1}(\by)$ and $\bz := H (\bx)$. Since $H(\bbo) = \bbo$, 
we find
$$
\frac{1}{A_H \cdot |\by|} \; \leq \; |\bz| \; \leq \; 
\frac{A_H}{|\by|}
$$
If $\bx$ is a point of $\bU_q$ at which $H$ is differentiable, we find 
the following estimate
$$
\|D_\by \iota (H)\| \; \leq \; \frac{1}{|\bz|^2}\cdot A_H \cdot 
\frac{1}{|\by|^2}
\; \leq \; A_H^3.
$$
Since $\iota_q$ is a $C^\infty$ diffeomorphism, the subset 
$\iota_q(\Rq \setminus \cD(H))$ has null measure. Thus $\iota(H)$ is 
differentiable outside a subset of zero measure with uniformly bounded
first derivatives. Thus it is Lipschitz.
\end{proof}
Let $\omg$ be either $\bbo$ or $\infty$. Let $h :(\Rqbar,\omg) \to 
(\Rqbar,\omg)$ be a germ of homeomorphism which is bi-Lipschitz over
$(\bU_q,\omg)$. Let $\omg^*$ be the point of $\Rqbar$ antipodal to 
$\omg$, that is
$$
\{\omg,\omg^*\} = \{\bbo,\infty\}.
$$
The map germ $\iota_q\circ h \circ\iota_q^{-1} : (\bU_q,\omg^*) \to 
(\bU_q,\omg^*)$ extends as a homeomorphism germ $\iota(h) : (\Rqbar,
\omg^*) \to (\Rqbar,\omg^*)$. A consequence of Proposition 
\ref{prop:inversion-global} is the (now expected) following result, 
initial motivation of the paper:
\begin{corollary}\label{cor:inversion-global-germs}
The germ of homeomorphism $\iota(h)$ is bi-Lipschitz over
$(\bU_q,\omg^*)$.
\end{corollary}
\begin{remark}\label{rmk:rademacher}
The proof of Proposition \ref{prop:inversion-global} we gave uses Rademacher 
Theorem, and is a direct proof. But this result is a special case of Lemma 
\ref{lem:main}, whose demonstration, although longer and mostly by absurd, 
uses even more elementary arguments. 
\end{remark}
%
%
%
%
%
%
%
%
%
%
%
%
%
%
%
%
%
%
%
%
%
%
%
%
%
%
%
%
%
%
%
%
%
%
%
%
%
%
%
%
%
%
\section{Inversion and germs of bi-Lipschitz homeomorphisms at infinity }\label{section:IaGobLHaI}

Let $\sgm$ be a point of $\Rqbar$. Following Definition \ref{def:q-affine} 
the notion of germ of $q$-affine subset at $\sgm$ is well defined. If $\tau$ is
a point of $\Rpbar$, the notion of Lipschitz mapping of affine germs $(S,\sgm) 
\to (T,\tau)$ is also well defined by Definition \ref{def:affine-lip}.

\bigskip
Let $\phi:(Y_1,\infty) \to (Y_2,\infty)$ be a germ of bi-Lipschitz homeomorphism
between $q_i$-affine subsets germs $(Y_i,\infty)$ with $i=1,2$. 
There exists a positive constant $A_\phi$ such that
$$
\by,\by'\in Y_1 \; \Longrightarrow \; 
\frac{1}{A_\phi} \cdot|\by'-\by| \, \leq 
\, |\phi(\by') - \phi(\by)| \, \leq \, A_\phi \cdot|\by' - \by|,
$$
Thus we can assume that the Lipschitz constant $A_\phi$ is such that the following estimates are 
also satisfied 
$$
\by \in Y_1 \; \Longrightarrow \;  \frac{1}{A_\phi} \cdot|\by| \, \leq 
\, |\phi(\by)| \, \leq \, A_\phi \cdot|\by|.
$$
With the notations of Section \ref{subsec:notations}, we find $\phi(\by) \sim \by$. 
\\
For $i=1,2$, let $X_i$ be the closure $\clos(\iota_{q_i}(Y_i))$ of $\iota_{q_i}(Y_i)$ 
in $\Rq$. The \em inversion of $\phi$ \em is the mapping defined as follows  
$$
\iphi : (X_1,\bbo) \to (X_2,\bbo), \;\; \bx \to \iphi(\bx) := 
\left\{
\begin{array}{rcl}
\iota_{q_2} \circ \phi \circ \iota_{q_1}^{-1}(\bx) & {\rm if} & \bx \in X_1^*
\\
\bbo & {\rm if} & \bx = \bbo
\end{array}
\right.
$$
It is 
clearly a germ of homeomorphism which extends continuously at $\bbo$ taking
the value $\bbo$ at $\bbo$. 
The homogeneity of the Euclidean metric as well
as the existence of the inversion mapping yield the following result.
%
%
\begin{lemma}\label{lem:magic-infinity}
If $\phi:(Y_1,\infty) \to (Y_2,\infty)$ is a bi-Lipschitz homeomorphism germ between 
$q_i$-affine subsets germs $(Y_i,\infty)$, for $i=1,2$, then its inversion 
$\iphi : (X_1,\bbo) \to (X_2,\bbo)$ is bi-Lipschitz homeomorphism germ.
\end{lemma}
\begin{proof}
First let us denote $h := \phi \circ \iota_{q_1}^{-1}$, that is
$$
h(\bx) = \phi\left(\frac{\bx}{|\bx|^2}\right).
$$
Therefore we get that
$$
\iphi(\bx) = \frac{h(\bx)}{|h(\bx)|^2}.
$$
Since $|h(\bx)| \sim |\bx^{-1}|$, we observe that $\iphi(\bx) \sim \bx$, 
more precisely:
$$
\frac{1}{A_\phi^3}\cdot |\bx| \, \leq \, |\iphi(\bx)| \,  \leq \, A_\phi^3 
\cdot|\bx|.
$$
It is sufficient to show that $\iphi$ is Lipschitz.

\medskip
Assume that $\iphi$ is not 
Lipschitz. Therefore there exist two sequences 
$(\bx_n)_n$ and $(\bx_n')_n$ of $\bU_{q_1}$ such that 
\begin{equation}\label{eq:not-lip}
\lim_n \frac{|\iphi(\bx_n) - \iphi(\bx_n')|}{|\bx_n - \bx_n'|} = \infty.
\end{equation}
We work with a representative of $\phi$ outside a compact subset $C_1$ of 
$\R^{q_1}$ containing $\bbo$ and with the representative of $\iphi$ over 
$X_1$, 
the closure $\clos(\iota_{q_1}(Y_1 \setminus C_1))$. Thus $X_1$ is 
compact.    
For convenience sake let 
$$
\by_n := \iota_{q_1} (\bx_n) \;\; {\rm and} \;\; \by_n' := 
\iota_{q_1} (\bx_n').
$$
We further define the following numbers:
$$
e_n := |\bx_n - \bx_n'|, \;\; 
\iphi_n := |\iphi(\bx_n) - \iphi(\bx_n')|, 
\;\; t_n := |\bx_n|, \;\;  t_n' := |\bx_n'|, 
\;\; s_n := |\by_n| \;\;{\rm and} \;\; s_n' := |\by_n'|.
$$
Of course we have $s_n t_n = s_n' t_n' = 1$. 

\smallskip Without loss of generality, we can assume that the sequence 
$(\bx_n)_n$ converges to $\chi \in X_1$ and $(\bx_n')_n$ converges to 
$\chi'\in X_1$.

\smallskip\noindent
$\bullet$ {\bf Case 1.} \em $\chi \neq \bbo$ and $\chi' \neq \bbo$. \em

In other words there exists a compact subset $K_1$ of $\bU_{q_1}$ which
contains $\bx_n,\bx_n'$ for all $n$. Since the inversion $\iota_{q_1}$ is 
bi-Lipschitz over $K_1$, so is the mapping $\iphi$, contradicting the 
estimate  \eqref{eq:not-lip}.
Therefore this case cannot happen and we can assume that $\chi = \bbo$.

\smallskip\noindent
$\bullet$ {\bf Case 2.} \em $\chi = \bbo$ and $\chi'\neq \bbo$. \em

Observe that the following estimates hold true
$$
e_n = t_n' + o(t_n'), \;\; |\iphi(\bx_n)| \sim t_n \to 0, \;\;
{\rm and} \;\; |\iphi(\bx_n')| \sim t_n' \in [a,b] 
$$ 
for positive real numbers $b>a$. Therefore we deduce that 
$$
\iphi_n \sim t_n'
$$
contradicting the estimate \eqref{eq:not-lip}. Therefore this case cannot 
happen and thus $\chi' = \bbo$ as well.

\smallskip\noindent
$\bullet$ {\bf Case 3.} \em $\chi = \chi' = \bbo$ and there exists $B>1$
such that $|\bx_n'|\geq B|\bx_n|$ for $n$ large enough. \em

For $n$ large enough, Claim \ref{claim:lip-1} yields 
$$
\frac{e_n}{t_n'} \in \left[\frac{B-1}{B},\frac{B+1}{B}\right].
$$
Since $\iphi_n \leq (1+B) A_\phi^3 t_n'$ for large $n$, we produce again
a contradiction to the estimate \eqref{eq:not-lip}. This case does not
occur and we can assume, up to taking a subsequence that $\frac{t_n}{t_n'} 
\to 1$ as $n$ goes to $\infty$.

\smallskip\noindent
Let $2\tht_n\in [0,\pi]$ be the angle between the vectors $\bx_n$ and $\bx_n'$.

\smallskip\noindent
$\bullet$ {\bf Case 4.} \em $\lim_n\frac{|\bx_n|}{|\bx_n'|} = 1$ and 
$\liminf_n 2\tht_n \in]0,\pi]$. \em

We can assume that $2\tht_n \geq 2\tht \in ]0,\pi]$. From Identity 
\eqref{eq:law-cosine} we deduce that
$$
\lim_n\frac{e_n}{t_n} \geq 2 \sin \tht > 0.
$$
Since $\iphi_n \leq A_\phi^3(t_n+t_n')$ for $n$ large enough, estimate 
\eqref{eq:not-lip} cannot be satisfied and thus $\tht = 0$.

Up to passing to subsequences,  we can assume that $(\tht_n)_n$ converges to $0$. 

\smallskip\noindent
$\bullet$ {\bf Case 5.} \em $\lim_n\frac{|\bx_n|}{|\bx_n'|} = 1$ and 
$\lim_n \tht_n =0$. \em
 
We can assume that $t_n' \geq t_n$ and $\bx_n' = \bx_n + \bz_n$ so that
$$
\ups_n :=\frac{|\bz_n|}{t_n'} = s_n' e_n\to 0 \;\; {\rm as} \;\;  n \to \infty, 
$$
Let $a_n := \cos(\tht_n)$, $b_n:= \sin(\tht_n)$ and $\dlt_n t_n := t_n' - t_n$. 
We get
$$
|\bz_n|^2 = e_n^2 = (2b_n t_n + \dlt_n t_n b_n)^2 + (\dlt_n t_n a_n)^2 
= t_n^2 \cdot[\dlt_n^2 + 4b_n^2 + o(b_n^2)].
$$
Since we can write $\by_n = \by_n' + \bw_n$, Equation \eqref{eq:e-E} 
yields $$
|\bw_n| = s_ns_n' \cdot |\bz_n|
$$
and thus we deduce that
$$
\frac{|\bw_n|}{s_n} = \frac{|\bz_n|}{t_n'} = \ups_n
\to 0 \;\; {\rm as} \;\;  n \to \infty.
$$
Since $\phi$ is bi-Lipschitz we obtain the following estimate
$$
h(\bx_n) = h(\bx_n') + s_n\bu_n, \;\; {\rm where} \;\;  |\bu_n| \sim \ups_n,
$$
from which we deduce 
$$
\frac{|h(\bx_n')|^2}{|h(\bx_n)|^2} \, = \, 1 + r_n
\;\; {\rm where} \;\; |r_n| \; \lsm \; \ups_n.
$$
Combining the various previous estimates yields the following one
$$
\iphi_n = \frac{1}{|h(\bx_n')|^2}\left|h(\bx_n') - (1+r_n)[h(\bx_n') + 
s_n\bu_n]\right| 
\;\lsm\; \frac{|r_n| + (1+r_n)|\bu_n|}{s_n} \sim t_n' \ups_n = e_n, 
$$
contradicting estimate \eqref{eq:not-lip}. 
\end{proof}
%
%
%
%
%
%
%
%
%
%
%
%
%
%
%
%
%
%
%
%
%
%
%
%
%
%
%
\medskip
\section{Inversion and germs of bi-Lipschitz homeomorphism at $\bbo$ }\label{section:IaGobLHa0}

\medskip
This section is about the counterpart at the origin of the previous result at 
infinity Lemma \ref{lem:magic-infinity}, more precisely its converse.

\medskip
Let $\psi:(X_1,\bbo) \to (X_2,\bbo)$ be a germ of bi-Lipschitz homeomorphism
between $q_i$-affine subsets $X_i$, where $i=1,2$.
Thus, there exists a positive constant $A_\psi$ such that
$$
\bx,\bx' \; \Longrightarrow \; \frac{1}{A_\psi} \cdot |\bx'-\bx| \, \leq \, 
|\psi(\bx') - \psi(\bx)| \, \leq \, A_\psi \cdot |\bx' - \bx|.
$$
Thus the following estimates are also satisfied 
$$
\bx \in X_1 \; \Longrightarrow \;  \frac{1}{A_\psi} \cdot|\bx| \, \leq 
\, |\psi(\bx)| \, \leq \, A_\psi \cdot|\bx|,
$$
that is $\psi(\bx) \sim \bx$, with the previous notation. 
Denoting $Y_i := \iota_{q_i} (X_i^*)$ for $i=1,2$, the \em inversion of $\psi$ \em is the germ of mapping defined as follows:
$$
\ipsi : (Y_1,\infty) \to (Y_2,\infty), \;\; \by \to \ipsi(\by) := 
\iota_{q_2} \circ \psi \circ \iota_{q_1}^{-1}(\by).
$$
It clearly extends as germ of homeomorphism $(\overline{Y}_1,\infty) \to (\overline{Y}_2,\infty)$. 

\medskip
The converse of Lemma \ref{lem:magic-infinity} is 
\begin{lemma}\label{lem:magic-origin}
If $\psi:(X_1,\bbo) \to (X_2,\bbo)$ is a bi-Lipschitz homeomorphism germ between 
$q_i$-affine subsetsgerms  $(X_i,\bbo)$, for $i=1,2$, then its inversion $\ipsi : 
(Y_1,\infty) \to (Y_2,\infty)$ is a bi-Lipschitz homeomorphism germ.
\end{lemma}
The proof will be symmetric to that of Lemma \ref{lem:magic-infinity} in the 
sense that arguments at $\bbo$ are replaced by their exact analogues at 
$\infty$, as expected from such a statement.
\begin{proof}

First, let $g := \psi \circ \iota_{q_1}^{-1}$, that is
$$
g(\by) = \psi\left(\frac{\by}{|\by|^2}\right).
$$
Therefore we get that
$$
\ipsi(\by) = \frac{g(\by)}{|g(\by)|^2},
$$
and since $|g(\by)| \sim |\by|^{-1}$ we find that $\ipsi(\by) \sim \by$, more
precisely
$$
\frac{1}{A_\psi^3} \cdot |\by|\, \leq \, |\ipsi(\by)| \,  \leq \, A_\psi^3 
\cdot |\by|.
$$
As in the previous section, it is enough to show that $\ipsi$ is Lipschitz.

\medskip
Assume that $\ipsi$ is not Lipschitz. Therefore there exist two sequences
$(\by_n)_n$ and $(\by_n')_n$ of $\bU_{q_1}$ such that 
\begin{equation}\label{eq:not-lip-0}
\lim_n \frac{|\ipsi(\by_n) - \ipsi(\by_n')|}{|\by_n - \by_n'|} = \infty.
\end{equation}
We work with a representative of $\psi$ within a compact subset $K_1$ of 
$\R^{q_1}$ containing $\bbo$. Let
$$
\bx_n := \iota_{q_1} (\by_n) \;\; {\rm and} \;\; \bx_n' := \iota_{q_1} (\by_n').
$$
In order to ease computations, we further define the following numbers:
$$
E_n := |\by_n - \by_n'|, \;\; \ipsi_n := |\ipsi(\by_n) - \ipsi(\by_n')|, 
\;\; s_n := |\by_n|, \;\;  s_n' := |\by_n'|, 
\;\; t_n := |\bx_n| \;\;{\rm and} \;\; t_n' := |\bx_n'|.
$$
Of course we find again that $s_n t_n = s_n' t_n' = 1$. 

\smallskip\noindent
$\bullet$ {\bf Case 1.} $\limsup_n \max (|\by_n|,|\by_n'|) \, < \infty$.

In other words there exists a compact subset $C_1$ of $\bU_{q_1}$ which
contains $\by_n,\by_n'$ for all $n$. Since the inversion $\iota_{q_1}$ is 
bi-Lipschitz over $C_1$, so is the mapping $\ipsi$, contradicting the 
estimate \eqref{eq:not-lip-0}.
Therefore this case cannot happen and we can assume, after taking a  
subsequence, that $(\by_n')_n$ converges to $\infty$.

\smallskip\noindent
$\bullet$ {\bf Case 2.} \em $\by_n' \to \infty$ and $\limsup_n|\by_n|<\infty$.
\em

Observe that the following estimates hold true
$$
E_n = s_n' + o(s_n'), \;\; |\phi(\by_n')| \sim s_n' \to \infty, \;\;
{\rm and} \;\; \ipsi(\by_n)| \sim s_n \in [a,b] 
$$ 
for positive real numbers $b>a$. Therefore we deduce that 
$$
\ipsi_n \sim s_n'
$$
contradicting the estimate \eqref{eq:not-lip-0}. This case cannot
happen and therefore a subsequence of $(\by)_n$ converges to $\infty$ as 
well.

\smallskip\noindent
$\bullet$ {\bf Case 3.} \em $\by_n,\by_n' \to \infty$ and there exists $B>1$
such that $|\by_n'|\geq B|\by_n|$ for $n$ large enough. \em

For $n$ large enough, Claim \ref{claim:lip-1} yields 
$$
\frac{E_n}{s_n'} \in \left[\frac{B-1}{B},\frac{B+1}{B}\right].
$$
Since $\ipsi_n \leq (1+B)A_\psi^3 s_n'$ we produce again a contradiction to 
the estimate 
\eqref{eq:not-lip-0}. This case does not occur and we can assume, up to taking  
subsequences that $\frac{s_n}{s_n'} \to 1$ as $n$ goes to $\infty$.

\medskip
Let $2\tht_n\in [0,\pi]$ be the angle between the vectors $\by_n$ and $\by_n'$.

\smallskip\noindent
$\bullet$ {\bf Case 4.} \em $\lim_n\frac{|\by_n|}{|\by_n'|} = 1$ and 
$\liminf_n 2\tht_n \in]0,\pi]$. \em

We can assume that $2\tht_n$ converges to $2\tht \in ]0,\pi]$. We check that
$$
\lim_n\frac{E_n}{s_n} \geq 2 \sin \tht > 0.
$$
Since $\ipsi_n \leq A_\psi^3 (s_n+ s_n')$ for $n$ large enough, estimate 
\eqref{eq:not-lip-0}
cannot be satisfied and thus $\tht = 0$. 

\smallskip\noindent
$\bullet$ {\bf Case 5.} \em $\lim_n\frac{|\by_n|}{|\by_n'|} = 1$ and 
$\lim_n \tht_n =0$. \em

We can assume that $s_n \geq s_n'$ and $\by_n = \by_n' + \bw_n$ so that
$$
\ups_n :=\frac{|\bw_n|}{s_n} = t_n E_n\to 0 \;\; {\rm as} \;\;  n \to \infty, 
$$
Let $a_n := \cos(\tht_n)$, $b_n:= \sin(\tht_n)$ and $\dlt_n s_n' := s_n - s_n'$. 
We get
$$
|\bw_n|^2 = E_n^2 = (2b_n s_n' + \dlt_n s_n' b_n)^2 + (\dlt_n s_n' a_n)^2 
= [\dlt_n^2 + 4b_n^2 + o(b_n^2)] (s_n')^2.
$$
Since we can write $\bx_n' = \bx_n + \bz_n$, Equation \eqref{eq:e-E} 
yields 
$$
|\bz_n| = t_n t_n' \cdot |\bw_n|
$$
and thus we deduce that
$$
\frac{|\bz_n|}{t_n'} = \frac{|\bw_n|}{s_n} = \ups_n 
\to 0 \;\; {\rm as} \;\;  n \to \infty.
$$
Since $\psi$ is bi-Lipschitz we obtain the following estimate
$$
g(\by_n) = g(\by_n') + t_n'\bu_n, \;\; {\rm where} \;\;  |\bu_n| \sim \ups_n,
$$
from which we deduce 
$$
\frac{|g(\by_n')|^2}{|g(\by_n)|^2} \, = \, 1 + r_n \;\; {\rm where} \;\; |r_n| 
\; \lsm \; \ups_n.
$$
Combining the various previous estimates yields the following one
$$
\ipsi_n = \frac{1}{|g(\by_n')|^2}|g(\by_n') - (1+r_n)(g(\by_n') + t_n'\bu_n)| 
\; \lsm \; \frac{|r_n| + (1+r_n)|\bu_n|}{t_n'} \sim s_n \ups_n = E_n, 
$$
contradicting estimate \eqref{eq:not-lip-0}. 
\end{proof}
%
%
%
%
%
%
%
%
%
%
%
%
%
%
%
%
%
%
%
%
%
%
%
%
%
%
%
%
%
\section{Inversion and bi-Lipschitz homeomorphisms}\label{section:IabLH}

This section presents the inversion lemma, that is the main tool of this 
note. It is a rather straightforward consequence of Lemma 
\ref{lem:magic-origin} and Lemma \ref{lem:magic-infinity}.


%
%
\begin{lemma}\label{lem:main}
Let $\Phi:W_1 \to W_2$ be a homeomorphism between $q_i$-affine subsets $W_i$, 
for $i=1,2$. Assume furthermore that, either $W_i$ contains the origin 
$\bbo\in \R^{q_i}$ for $i=1,2$ and $\Phi(\bbo) = \bbo$, or $\bbo\notin 
W_i$ for $i=1,2$. The mapping defined as 
$$
\iPhi : \iota_{q_1}(W_1^*) \to 
\iota_{q_2}(W_2^*), \;\; \bx \to 
\iPhi(\bx) := \iota_{q_2} \circ \Phi \circ \iota_{q_1}^{-1}(\bx).
$$
is bi-Lipschitz if and only if $\Phi$ is. 
Moreover if $W_1$ is unbounded, the mapping $\iPhi$ extends 
bi-Lipschitz-ly at $\bbo$, taking the value $\bbo$.
\end{lemma}
\begin{proof} Since $\iota(\iPhi) = \Phi$, it is enough to show the result 
when $\Phi$ is bi-Lipschitz. Since $\iPhi^{-1} = \iota_{q_1} \circ 
\Phi^{-1} \circ \iota_{q_2}^{-1}$, we only need to show that $\iPhi$ is 
Lipschitz.

\smallskip
By construction we already know that $\iPhi$ is a homeomorphism of 
$\iota_{q_1}(W_1^*)\to 
\iota_{q_2}(W_2^*)$ which extends homeomorphically to 
$\overline{\iota_{q_1}(W_1^*)}\to 
\overline{\iota_{q_2}(W_2^*)}$, mapping $\bbo$ to $\bbo$ whence $W_1$ is
unbounded.

Let $0<r<R<\infty$ be radii. We define the following subsets: 
\begin{eqnarray*}
C_r^* & := & \{\by \in \iota_{q_1}(W_1^*) : 0 < |\by| \leq r\} \\
C_r^R & := & \{\by \in \iota_{q_1}(W_1^*) : r\leq |\by| \leq R\} \\
C^R & := & \{\by \in \iota_{q_1}(W_1^*) : R\leq |\by|\} 
\end{eqnarray*}
The "annulus" $C_r^R$ is compact and does not contain the  origin. We
recall that the mapping $\iota_q$ induces a bi-Lipschitz homeomorphism 
$K \to \iota_q(K)$ over any compact subset $K$ of $\bU_q$. Thus $\iPhi$ 
induces a bi-Lipschitz homeomorphism from $C_r^R$ onto its image. 
By Lemma \ref{lem:magic-origin}, $\iPhi$ is a bi-Lipschitz homeomorphism
from $C^R$ onto its image. If $W_1$ is compact, we can take $r$ small enough
so that $C_r^*$ is empty. 
If $W_1$ is unbounded, Lemma \ref{lem:magic-infinity} implies that 
$\iPhi$ is a bi-Lipschitz homeomorphism from $C_r^*$ onto its image, and 
extends bi-Lipschitz-ly at $\bbo$.

Let $A,B \in \{C^R,C_r^R,C_r^*\}$ with $A\neq B$.
If $\iPhi$ is not Lipschitz in $A\cup B$, then there exist a pair of sequences
$(\ba_n)_n$ of $A$ and $(\bb_n)_n$ of $B$ such that 
$$
\lim_{n\to\infty}\frac{|\iPhi(\ba_n) - \iPhi(\bb_n)|}{|\ba_n - \bb_n|}
=\infty.
$$
Up to taking some sub-sequences we can further assume that each sequence 
$(\ba_n)_n$ and $(\bb_n)_n$ either converges or goes to infinity. 
We check that the only scenario where such a pair of sequences could exist 
satisfying the required limit condition above, is when both converge to
a point $\bc \in A \cap B$, thus $\bc \neq \bbo$. 
Which is absurd since nearby the point $\bc$ the 
inversion $\iota_{q_1}^{-1}$ is Lipschitz as is $\iota_{q_2}$ near the point 
$\Phi(\iota_{q_1}^{-1}(\bc))$.
\\
Thus $\iPhi$ is Lipschitz over each $A\cup B$ with $A,B \in 
\{C^R,C_r^R,C_r^*\}$ with $A\neq B$. Let $C$ be the remaining subset
among $\{C^R,C_r^R,C_r^*\}$ so that $A\cup B\cup C = \iota_{q_1}(W_1)$.
Working with $A' = A\cup B$ and $B' = C$ instead of $A,B$ as we did in the 
previous case, we conclude that $\iPhi$ is Lipschitz
\end{proof}
%
%
%
%
%
%
%
%
%
%
%
%
%
%
%
%
%
%
%
%
%
%
%
%
%
%
%
%
%
%
%
%
%
%
%
%
%
%
\section{Main result}\label{section:MR}
$ $ 
Let $N_q = (0,\ldots,0,1) \in \R^{q+1}$ be the north pole of the unit sphere
$\bS^q$. Let
$$
\sgm_q :\Rq\to \bS^q\setminus N_q, \;\;
\bx \; \mapsto \; \left(\frac{2\bx}{1+|\bx|^2},\frac{|\bx|^2-1}{|\bx|^2+1}
\right)
$$ 
be the inverse of the stereographic projection with centre $N_q$. Given a 
subset of $S$ of $\Rq$, let 
$$
\wtS := \clos(\sgm_q(S)).
$$

\medskip
If $\Phi:W_1 \to W_2$ is a homeomorphism, where each subset $W_i$ is 
$q_i$-affine, $i=1,2$, its stereographic pre-compactification is the homeomorphism 
$\sgm_{q_2}\circ\Phi\circ\sgm_{q_1}^{-1}$. If $W_1$ is unbounded, the stereographic 
compactification of $\Phi$ is the mapping $\wtPhi$, extension of 
$\sgm_{q_2}\circ\Phi\circ\sgm_{q_1}^{-1}$ to $\wtW_1$.

\begin{theorem}\label{thm:main}
Let $W_i$ be $q_i$-affine subsets, $i=1,2$. A mapping 
$\Phi:W_1\to W_2$ is bi-Lipschitz, if and only if its stereographic compactification 
$\wtPhi: \wtW_1 \to \wtW_2$ is bi-Lipschitz
\end{theorem}
The rest of this section is devoted to the proof of this result. 

\medskip
We recall that the quotient space obtained from gluing two copies 
of $\Rq$, when both $\bU_q$ are identified by the 
inversion $\iota_q$ is $\bS^q$. 
Therefore the next result, somehow tuned with
\cite[Lemma 7.2]{CoGrMi2}, should not come as a surprise.
\begin{lemma}\label{lem:inversion-north}
Suppose that $W_1$ is unbounded. 
The mapping germ $\Phi: (W_1,\infty) \to (W_2,\infty)$ is bi-Lipschitz if and 
only if $\wtPhi :(\wtW_1,N_{q_1}) \to (\wtW_2,N_{q_2})$ is bi-Lipschitz.
\end{lemma}
\begin{proof}
Let $\bz = (\bz',t)$ be Euclidean coordinates on $\R^{q+1} = \Rq\times\R$.
The following mapping  
$$
\beta_q : \bB_\frac{1}{2}^q \to \bS^q \cap \left\{ t \geq \frac{3}{5}
\right\}, \;\; \by \to \left( \frac{1}{1+|\by|^2} \cdot \by, 
\frac{1-|\by|}{1 + |\by|^2} \right)
$$
is a $C^\infty$ diffeomorphism, thus is a bi-Lipschitz homeomorphism
mapping $\bbo$ onto $N_q$.
We also check that
$$
|\bx| \geq 2 \; \Longrightarrow \; \beta_q\circ\iota_q(\bx) = \sgm_q(\bx).
$$
The Lemma follows from Lemma \ref{lem:magic-origin}, Lemma 
\ref{lem:magic-infinity} and the fact that $\beta_q$ is bi-Lipschitz.
\end{proof}
\begin{proof}[Proof of Theorem \ref{thm:main}]
We recall that $\sgm_q$ is bi-Lipschitz over any given compact subset 
of $\Rq$. If $W_1$ is compact the result is thus obvious.

\smallskip
Assume that $W_1$ is unbounded. 
Let $K_1 = \bB_{R_1}^{q_1}\cap W_1$ and 
$K_2 = \Phi(K_1)$ with $R_1\geq 2$ chosen so that $W_i \setminus K_i$
is contained in $\R^{q_i}\setminus \bB_2^{q_i}$, where $i=1,2$. 
Thus the mapping 
$$
\Phi_b:=\Phi|_{K_1} : K_1 \to K_2
$$
is bi-Lipschitz, if and only if $\wtPhi_b := \wtPhi|_{\wtK_1} : \wtK_1 
\to\wtK_2$ is bi-Lipschitz. 

Up to increasing $R_1$, following the proof of Lemma 
\ref{lem:inversion-north}, we deduce that the mapping 
$$
\Phi_u:= \Phi|_{W_1\setminus K_1} :W_1\setminus K_1 \to W_2\setminus K_2
$$
is bi-Lipschitz if and only if the mapping 
$\wtPhi_u := \wtPhi|_{\wtW_1\setminus \wtK_1} : \wtW_1\setminus \wtK_1 \to
\wtW_2\setminus \wtK_2$ is. Observe that $\Phi_u$, $\wtPhi_u$ respectively, extends
bi-Lipschitz-ly on the closure of its domains when $\Phi$, 
$\wtPhi$ respectively, is bi-Lipschitz. 

\smallskip\noindent
$\bullet$ Assume that $\Phi$ is bi-Lipschitz. Thus $\wtPhi$ is a 
homeomorphism and both $\wtPhi_b$ and $\wtPhi_u$ are bi-Lipschitz. 

If $\wtPhi$ were not Lipschitz, there would exist two sequences 
of $(\bz_n)_n$, $(\bz_n')_n$ of $\wtW_1$ such that 
$$
\lim_{n\to\infty}\frac{|\wtPhi(\bz_n') - \wtPhi(\bz_n)|}{|\bz_n' - \bz_n|} 
= \infty. 
$$
Since $\wtW_1$ is compact, up to passing to subsequences we can assume
that both sequences converge to $\omg_1$. Since $\Phi_u,\Phi_b$ 
are bi-Lipschitz and $\Phi_u$ extends bi-Lipschitz-ly onto 
$\clos(\wtW_1\setminus \wtK_1)$, necessarily one of the sequences is 
contained in $\wtK_1$ and the other one in $\wtW_1\setminus \wtK_1$. Thus 
$\omg_1 \in \wtK_1$ and $\omg_2 := \wtPhi(\omg_1) \in \wtK_2$. Since 
$\sgm_{q_1}^{-1}$ is bi-Lipschitz nearby $\omg_1$ and  $\sgm_{q_2}^{-1}$ is 
bi-Lipschitz nearby $\omg_2$, the 
mapping $\wtPhi$ is Lipschitz nearby $\omg_1$, yielding a contradiction.

\smallskip\noindent
$\bullet$ Assume that $\wtPhi$ is bi-Lipschitz. Thus $\Phi$ is a 
homeomorphism and both $\Phi_b$ and $\Phi_u$ are bi-Lipschitz. 
Moreover $\Phi_u$ extends bi-Lipschitz-ly to $\clos(W_1\setminus K_1)$.

If $\Phi$ were not Lipschitz, there would exist two sequences 
of $(\bx_n)_n$, $(\bx_n')_n$ of $\wtW_1$ such that 
$$
\lim_{n\to\infty}\frac{|\Phi(\bx_n') - \wtPhi(\bx_n)|}{|\bx_n' - \bx_n|} 
= \infty. 
$$
Necessarily one sequence belongs to $K_1$ and the other one to 
$\clos(W_1\setminus K_1)$. Assume that $(\bx_n)_n$ is contained in $K_1$.
So we can assume it converge to $\by_1$. If $\by_1$ does not lie
in the compact set $L_1 = K_1\cap \clos(W_1\setminus K_1)$, thus  
$$
\liminf_n |\bx_n' - \bx_n|  \in (0,\infty]
$$
therefore $\Phi(\bx_n')$ goes to $\infty$, and using $\wtPhi$, we conclude
that $\bx_n' \to \infty$. 
Let $M$ be a Lipschitz constant common to $\Phi_b$ and $\Phi_u$.
Let $\by_1'$ be a point of $L_1$.
Thus
$$
|\Phi(\bx_n') - \Phi(\bx_n)|\leq |\Phi_u(\bx_n) - \Phi_u(\by_1')| + 
|\Phi_b(\by_1')-\Phi_b(\bx_n)| \leq M|\bx_n - \by_1'| + M |\by_1'-\bx_n|
$$
yielding a contradiction since $|\bx_n-\by_1'|\to \infty$. 
Thus $\by_1$ lies in $L_1$.

The same argument involving the point $\by_1' = \by_1$ implies that 
$\liminf_n |\bx_n' - \bx_n|  = 0$, so we can assume that $(\bx_n')_n$
converge to $\by_1$ as well. Since $\sgm_{q_1}^{-1}$
is bi-Lipschitz nearby $\by_1$ and $\sgm_{q_2}$ is bi-Lipschitz nearby 
$\Phi(\by_1)$, the mapping $\Phi$ is Lipschitz nearby $\by_1$, yielding a 
contradiction. 
\end{proof}
%
%
%
%
%
%
%
%
%
%
%
%
%
%
%
%
%
%
%
%
%
%
%
%
%
%
%
%
%
%
%
%
%
%
%
%
%
%
%
%
\section{Geometry at infinity of tame sets}\label{section:GaITS}

There are many possible applications of the inversion Lemma \ref{lem:main}. 
In particular, any 
bi-Lipschitz classification problem of subsets at infinity is equivalent to
a bi-Lipschitz classification problem at the origin.

\medskip
There are quite a few questions of bi-Lipschitz definable geometry
at infinity which now reduce to a problem at the origin by our main result. 
Many of them  
would require some specific preparations, that is why we present here only two
such applications, which are immediate consequences of 
Lemma~\ref{lem:main}. 

\medskip
\subsection{Bi-Lipschitz definable sets at infinity and their tangent cones}
$ $

\smallskip
A non-negative cone $C$ of $\Rq$ is any subset of $\Rq$ stable by 
non-negative rescaling:
$$
\bx \in C \; \Longrightarrow \; t\cdot \bx \in C, \; \forall t\geq 0.
$$
For a given non-negative cone $C$ of $\Rq$, the \em link of $C$ \em is defined as
$$
\bS(C) := C\cap \bS^{q-1}.
$$
Let $S$ be a non-empty subset of $\Rq$. The non-negative cone over $S$ with vertex the 
origin $\bbo$ is the subset of $\Rq$ defined as
$$
\wt{S}^+:= \{t\bu \in \Rq : \bu \in S, \; t\geq 0\}.
$$
In particular, a subset $C$ is a non-negative cone of $\Rq$ if and only if it is
the non-negative cone over its link: 
$$
C = \wt{\bS(C)}^+. 
$$
\begin{definition}\label{def:asymptotic-set}
Let $S$ be a subset of $\Rq$. 

(i) The \em asymptotic
set of $S$ at $\bbo$ is the closed subset of the unit sphere $\bS^{q-1}$
$$
S^\bbo := \left\{\bu \in \bS^{q-1} \;  : \; \exists \; (\bx_k)_k \in S^* \;\;
{\rm such \: that} \;\; \bx_k\to \bbo \;\;{\rm and} \;\; 
\frac{\bx_k}{|\bx_k|} \to \bu
\right\}.
$$

(ii) The \em asymptotic
set of $S$ at $\infty$ is the closed subset of the unit sphere $\bS^{q-1}$
$$
S^\infty := \left\{\bu \in \bS^{q-1} \;  : \; \exists \; (\bx_k)_k \in S \;\;
{\rm such \: that} \;\; |\bx_k| \to \infty \;\;{\rm and} \;\; 
\frac{\bx_k}{|\bx_k|} \to \bu
\right\}.
$$
\end{definition}
The subset $S^\omg$ is not empty if and only if $\overline{S}$ contains 
$\omg$, and observe that 
$$
\clos(S)^\omg = S^\omg
$$
where $\omg =\bbo$ or $\infty$. Since we are interested in the non-negative cones
$\wt{S^\bbo}^+$ and $\wt{S^\infty}^+$, we will work only with closed subsets. 
The non-negative cone $\wt{S^\omg}^+$ is also known as the \em tangent cone of $S$
at $\omg$, \em for $\omg = \bbo$ or $\infty$.  

\bigskip
Given $\bx \in \bU_q$, observe the following obvious fact 
$$
\frac{\bx}{|\bx|} = \frac{\iota_q(\bx)}{|\iota_q(\bx)|}.
$$
Let $X$ be a closed subset of $\Rq$ and let $\iota(X)$ be the closure 
$\clos(\iota_q(X^*))$. The following result is obvious from the definitions of 
asymptotic sets and the inversion.
 \begin{lemma}\label{lem:inv-tang-0-infty}  
The following identities hold true:
$$
X^\infty = \iota(X)^\bbo \;\; {\rm and} \;\; X^\bbo = \iota(X)^\infty. 
$$
\end{lemma}
From this lemma we deduce 
\begin{equation}\label{eq:cones-0-infty}
\wt{X^\infty}^+ = \wt{\iota(X)^\bbo}^+ \;\; {\rm and} \;\; 
\wt{X^\bbo}^+ = \wt{\iota(X)^\infty}^+. 
\end{equation}

\bigskip
Let $\cM$ be a polynomially  bounded o-minimal structure expanding the real
field $(\R,+,.,\geq)$ (see \cite{vdD}). 
A subset of an Euclidean space $\Rq$ is \em definable \em
if it is definable in $\cM$. Let $S$ be a subset of $\Rq$. A mapping $S\to\Rp$
is \em definable \em if its graph is definable.

\medskip
We recall the following result of Sampaio about tangent cones:
\begin{theorem}[\cite{Sam}]
Let $(X_i,\bbo)$ be the germ of a definable set of $\R^{q_i}$ at the origin,
i=1,2. If there 
exists a bi-Lipschitz homeomorphism $(X_1,\bbo) \to (X_2,\bbo)$, then there
exists a bi-Lipschitz homeomorphism $\wt{X_1^\bbo}^+ \to \wt{X_2^\bbo}^+$ 
mapping $\bbo$ onto $\bbo$. 
\end{theorem}   
In truth \cite{Sam} deals only with sub-analytic subsets, but the part of the 
demonstration using sub-analyticity goes through the definable context 
readily. 

\bigskip
We recall that the inversion $\iota_q$ is a rational mapping, thus 
semi-algebraic, therefore definable in $\cM$. 
As a corollary of this latter fact, of the inversion Lemma \ref{lem:main} and 
of identity \eqref{eq:cones-0-infty}, we obtain
the following
\begin{proposition}\label{prop:tangent-cones-infty}
Let $(W_i,\infty)$ be the germ of a closed definable set of $\R^{q_i}$ at 
infinity, i=1,2. If there exists a bi-Lipschitz homeomorphism $(W_1,\infty) 
\to (W_2,\infty)$, then there exists a bi-Lipschitz homeomorphism 
$\wt{W_1^\infty}^+ \to \wt{W_2^\infty}^+$ mapping $\bbo$ onto $\bbo$. 
\end{proposition}   
\medskip
\subsection{On the link at infinity}
$ $

\smallskip
Let $S$ be a subset of $\Rq$. For any positive radius $R$, we define the 
following subsets:
$$
S_R := S\cap\bS_R^{q-1}, \;\; S_{\leq R} := S\cap \bB_R^q \;\; {\rm and} \;\;
S_{\geq R} := S\setminus B_R^{q}.
$$
Let again denote $\iota(S)$ the closure of $\iota_q(S^*)$. Thus we get
the obvious identifications
$$ 
\iota_q(S_R) = \iota(S)_\frac{1}{R}, \;\; 
\iota_q(S_{\leq R}^*) = \iota_q(S)_{\geq \frac{1}{R}} \;\; {\rm and} \;\;
\iota_q(S_{\geq R}) = \iota_q(S)_{\leq\frac{1}{R}}^*. 
$$

\bigskip
When $X$ is definable and contains the origin $\bbo$, the local conic 
structure Theorem states that \em there exists $r_0$ such that for any 
radius $r_0\geq r>0$, the definable subset $X_{\leq r}$ is definably 
homeomorphic with $( \wt{X_r}^+ )_{\leq r}$, the "non-negative cone over $X_r$" \em 
\cite{vdD}. 
Moreover such a definable homeomorphism can be found so that it 
preserves the distance to $\bbo$ \cite{Val1}. In particular $X_r$ has constant
topological type for $r \leq r_0$. It can also be shown that 
for any pair of radii $0<r<r'\leq r_0$, the links $X_r$ and $X_{r'}$ are 
bi-Lipschitz definably homeomorphic \cite{Val1,Val3,Val4}, although the 
Lipschitz 
constant in general cannot be uniform over $]0,r_0]$. 

\medskip
Let $W$ be a definable subset of $\R^q$. Using the inversion and the local conic structure Theorem yield the locally conic structure Theorem at infinity: \em There exists a 
positive radius $R_0$ such that for any $R\geq R_0$, the subset $W_{\geq R}$ is definably homeomorphic to $(\wt{W_R}^+)_{\geq R}$, the "non-negative cone over 
$W_R$" at infinity. \em Moreover such a definable homeomorphism can be found 
so that it also preserves the distance to the origin. Last given any pair of
radii $R,R'\geq R_0$, the links $W_R$ and $W_{R'}$ are definable and bi-Lipschitz 
homeomorphic.
 
\medskip
We have mentioned the local conic structure theorems and the bi-Lipschitz 
constancy of the links in light of the following result about Links preserving
reparametrization of definable bi-Lipschitz homeomorphism of Valette:
\begin{theorem}[\cite{Val2,Val3,Val4}]\label{thm:bilip-dist-pres}
Let $(X_i,\bbo)$ be a closed definable germ of $\R^{q_i}$, $i=1,2$. If there exists
a definable bi-Lipschitz homeomorphism $(X_1,\bbo) \to (X_2,\bbo)$, then there 
exists a definable bi-Lipschitz homeomorphism $(X_1,\bbo) \to (X_2,\bbo)$ 
preserving the distance to the origin.
\end{theorem}

Again as a Corollary of our main result Theorem \ref{thm:main} and Theorem
\ref{thm:bilip-dist-pres} and the semi-algebraicity of the inversion, thus 
definable in $\cM$, we find the following 
\begin{proposition}\label{prop:bilip-dist-pres-infty}
Let $(W_i,\infty)$ be a closed definable germ of $\R^{q_i}$, $i=1,2$. If 
there exists
a definable bi-Lipschitz homeomorphism $(W_1,\infty) \to (W_2,\infty)$, 
then there exists a definable bi-Lipschitz homeomorphism $(W_1,\infty) \to 
(W_2,\infty)$ preserving the distance to the origin.
\end{proposition}
%
%

%
%
%
%
%
%
%
%
%
%
%
%
%
%
%
%
%
%
%
%
%
%
%
%
%
%
%
%
%
%
%
%
%
%
%
%
%

%
%
\end{document}